\documentclass[11pt]{article}
\usepackage{amssymb,amsmath,amsthm,color,cite}
\usepackage[dvips]{graphicx}

\IfFileExists {timestamp.tex } {\pagestyle{myheadings}{\input timestamp.tex }} {}

\newtheorem{theorem}{Theorem}
\newtheorem*{conjecture*}{Conjecture}
\newtheorem{question}{Question}
\newtheorem*{question*}{Question}
\newtheorem{lemma}[theorem]{Lemma}

\newtheorem*{claim*}{Claim}

\theoremstyle{remark}
\newtheorem*{remark}{Remark}

\newcommand{\degree}{\ensuremath{^\circ}} 
\newcommand{\ar}{A_{(r)}} 
\newcommand{\dr}{D_{(r)}} 
\newcommand\eps{\varepsilon}

\newcommand\R{\mathbb{R}}

\newcommand\Prb{\mathbb{P}}
\newcommand\cP{{\mathcal P}}

\begin{document}
\title{Small components in $k$-nearest neighbour graphs
} \author{Mark
  Walters\footnote{School of Mathematical Sciences, Queen Mary,
University of London, London E1 4NS, England.  
{\tt m.walters@qmul.ac.uk}}
} \maketitle
\begin{abstract}
  Let $G=G_{n,k}$ denote the graph formed by placing points in a
  square of area $n$ according to a Poisson process of density 1 and
  joining each point to its $k$ nearest neighbours.
  In~\cite{MR2135151} Balister, Bollob\'as, Sarkar and Walters proved
  that if $k<0.3043\log n$ then the probability that $G$ is connected tends
  to 0, whereas if $k>0.5139\log n$ then the probability that $G$ is
  connected tends to 1.

 We prove that, around the threshold for connectivity, all vertices
 near the boundary of the square are part of the (unique) giant
 component. This shows that arguments about the connectivity of $G$
 do not need to consider `boundary' effects.

  We also improve the upper bound for the threshold for connectivity
  of $G$ to $k=0.4125\log n$.
\end{abstract}

\subsection{Introduction}
Let $S_n$ denote a $\sqrt n\times \sqrt n$ square and let
$G_{n,k}$ denote the graph formed by placing points
in $S_n$ according to a Poisson process $\cP$ of density 1 and joining
each point to its $k$-nearest neighbours by an undirected edge. Since
we shall be interested in the asymptotic behaviour of this graph as
$n\to\infty$, it is convenient to introduce one piece of notation. For
a graph property $\Pi$ we say that $G_{n,k}$ has $\Pi$ \emph{with high
  probability} (abbreviated to whp) if $\Prb(G_{n,k}\text{ has
  $\Pi$})\to 1$ as $n\to\infty$.

Xue and Kumar~\cite{xue-kumar} proved that the threshold for
connectivity is $\Theta(\log n)$; more precisely they showed that if
$k=k(n)>5.1774\log n$ then $G_{n,k}$ is connected whp, and if
$k=k(n)<0.074\log n$ then $G_{n,k}$ is whp not connected.

Subsequent work by Balister, Bollob\'as, Sarkar and
Walters~\cite{MR2135151} substantially improved the upper and lower
bounds to $0.5139\log n$ and $0.3043\log n$ respectively.  In their
proof they also showed that for any $k=\Theta(\log n)$ the graph
consists of a giant component containing a proportion $1-o(1)$ of all
vertices and (possibly) some other `small' components of (Euclidean)
diameter $O(\sqrt{\log n})$ (for a formal statement see
Lemma~\ref{l:small}).

Moreover, they showed that if $k>0.311\log n$ then $G$ has no small
component within distance $O(\sqrt{\log n})$ of the boundary of
$S_n$. Unfortunately, there is a gap between this bound and the lower
bound of $0.3043$ mentioned above. This means that close to the threshold
for connectivity the obstruction to connectivity could occur near the
boundary of the square or it could occur in the centre (their methods
did rule out the possibility that the obstruction occurs in the corner
of the square). This has caused several problems in later papers
(e.g., \cite{MR2514943}) where the authors had to consider both cases in
their proofs.

Our main result is the following theorem showing that, in fact, the
obstruction must occur away from the boundary of $S_n$. This should
simplify subsequent work in the area as only central components need
to be considered. (Of course, the improvement itself is only of minor
interest, it is the fact that the new upper bound for the existence of
components near the boundary is smaller than the general lower bound
that is of importance.)

\begin{theorem}\label{t:boundary}
Suppose that $G=G_{n,k}$ for some $k>0.272 \log n$. Then there is a
constant $\eps>0$ such that the probability that there exists a vertex within
distance $\log n$ of the boundary of $S_n$ that is not contained in the
giant component is $O(n^{-\eps})$.
\end{theorem}
\begin{remark}
  The distance $\log n$ to the boundary is much larger than the
  typical edge length and (non-giant) component sizes which are
  $O(\sqrt {\log n})$. Moreover, the theorem would still be true with
  $\log n$ replaced by a small power of $n$.
\end{remark}
\noindent%
Our second result is the following improvement on the upper bound
for connectivity of $G$.
\begin{theorem} \label{t:centre}
Suppose that $G=G_{n,k}$ for some $k>0.4125 \log n$. Then whp $G$ is
connected.
\end{theorem}
To illustrate Theorem~\ref{t:centre} let $D$ be a disc of radius $r$
and consider the event that there are $k+1$ points inside $D$ and no
points in $3D\setminus D$ (where $3D$ denotes the disc with same
centre as $D$ and three times the radius). If this event occurs then
the $k$-nearest neighbours of any point in $D$ also lie in $D$: in
particular, there are no `out'-edges from $D$ to the rest of the
graph.  If we choose $r$ such that $9\pi r^2\approx k+1$ (to maximise
the probability of this event) then the probability of a specific
instance of this event is about $9^{-(k+1)}$. Since we can fit
$\Theta(n/\log n)$ disjoint copies of this event into $S_n$ we see
that if $k< (\frac1{\log9}-\eps)\log n$ (for some $\eps>0$) then whp
this event occurs somewhere in $S_n$ and thus that $G$ has a subgraph with
no out-degree. Since $1/\log 9\approx0.455>0.4125$,
Theorem~\ref{t:centre} shows that there is a range of $k$ for which
the graph is connected whp but contains pieces with no outdegree.
(The corresponding result for in-degree was proved
in~\cite{MR2135151}.)

The proofs of these two theorems are broadly similar: they use the
ideas from \cite{MR2135151} but also consider points which are near
the small component but not contained in it. Indeed, if one looks at
the lower bound proved in \cite{MR2135151} we see that the density of
points near the small component is higher than average. This is an
unlikely event and we incorporate it into our bounds. Indeed, the
above observation that there are small pieces of the graph with no
out-degree shows that any proof of Theorem~\ref{t:centre} (or any
stronger bound) must consider points outside of a potential small
component and show that they send edges in.

The key step is to split into two regimes depending on whether there
is a point `close' to the small component.  If there is no such point
then the `excluded area' from the small component is quite large (which
is unlikely), whereas if there is such a point then it must have a small
$k$-nearest neighbour radius (which is also unlikely).

\subsection{Notation and Preliminaries}
We start with some notation.  For any point $x$ and real number $r$
let $D(x,r)$ denote the closed disc of radius $r$ about $x$.  We shall
also use the term \emph{half-disc of radius $r$ based at $x$} to mean one of
the four regions obtained by dividing the disc $D(x,r)$ in half vertically
or horizontally.

For a set $A$ in $S_n$ let $|A|$ denote the measure of $A$, and $\#A$
denote the number of points of $\cP$ in $A$. For any real number~$r$
let $\ar $ be the $r$-blowup of $A$ defined by \[\ar =\{x\in \R^2:
d(x,A)<r\}.\] Note that we do allow $\ar $ to contain points outside of
$S_n$.

Finally, whenever we use the term diameter we shall always mean the
\emph{Euclidean} diameter: we do not use graph diameter at any point
in the paper.

We shall need a few results from the paper of Balister, Bollob\'as,
Sarkar and Walters~\cite{MR2135151}. Since our notation is slightly
different we quote them here for convenience.  The first is a slight
variant of Lemma~6 of~\cite{MR2135151} which follows immediately from
the proof given there (see also Lemma~1 of~\cite{MR2514943}).
\begin{lemma}\label{l:small}
 For fixed\/ $c>0$ and\/ $L$, there exists $c_1=c_1(c,L)>0$, depending
 only on $c$ and\/~$L$, such that for any $k\ge c\log n$, the
 probability that\/ $G_{n,k}$ contains two components each of\/
 $($Euclidean$)$ diameter at least\/ $c_1\sqrt{\log n}$, is $O(n^{-L})$.
\end{lemma}
The second bounds the probability of a small component near one side,
or two sides of $S_n$; it is explicit in the proof of Theorem~7
of~\cite{MR2135151}. (Note, Theorem~\ref{t:boundary} improves the
first of these bounds.)
\begin{lemma}\label{l:boundary}
  Suppose that $k=\Theta(\log n)$. The probability that there is a
  small component containing a vertex within $\log n$ of one boundary
  of $S_n$ is $O(n^{\frac12+o(1)}5^{-k})$ and the probability that
  there is a small component containing a vertex within $\log n$ of
  two sides of $S_n$ is $O(n^{o(1)}3^{-k})$.
\end{lemma}

The final result follows easily from concentration results for the
Poisson distribution (see e.g.~\cite{MR2437651}) and most of it is
implicit in Lemma~2 of~\cite{MR2135151}.
\begin{lemma}\label{l:edges}
  For any fixed $c$ and $L$ there is a constant $c_2(c,L)$ such that
  for any $k$ with $c\log n<k<\log n$ the probability that there is
  any edge of length at least\/ $c_2\sqrt{\log n}$, or any two points
  within distance $\frac{1}{c_2}\sqrt{\log n}$ of each other not
  joined by an edge, or a point $x\in\cP$ with a half-disc of radius
  $c_2\sqrt{\log n}$ based at $x$ contained entirely inside $S_n$ that
  contains no points of $\cP$, is $O(n^{-L})$.
\end{lemma}

We will use the following simple but technical lemma several
times. 
\begin{lemma}\label{l:kk}
  Suppose that $A,B,C$ are three sets in $S_n$ with
  $|A|\le |C|$ and $|B|\le |C|$ then 
\[
\Prb(\text{$\#A\ge k$, $\#B\ge k$, $\#(A\cap B)=0$ and
  $\#C=0$})\le\left(\frac{4|A||B|}{(|A|+|B|+|C|)^2}\right)^k.
\]
\end{lemma}
\begin{proof}Let $A'=(A\setminus B)\setminus C$, $B'= (B\setminus A)\setminus C$, $C'=C\cup (A\cap B)$,
and $U=A\cup B\cup C=A'\cup B'\cup C$. We see that $A',B'$ and $C'$ are
pairwise disjoint so $|U|=|A'|+|B'|+|C'|$ and, since $\#(A\cap B)=0$,
that $\#A'\ge k$, $\#B'\ge k$.
We have
\begin{align*}
  \Prb(\#A\ge k&,\ \#B\ge k,\ \#(A\cap B)=0\text{ and }\#C=0)\\
&=  \Prb(\#A'\ge k,\ \#B'\ge k\text{ and }\#C'=0)\\
&=\sum_{l\ge k,m\ge
    k} \Prb(\#A'=l,\ \#B'= m\text{ and }\#U=l+m)\\
  &=\sum_{l\ge k,m\ge
    k} \Prb(\#A'=l,\ \#B'= m\ |\ \#U=l+m)\Prb(\#U=l+m)\\
  &\le \max_{l\ge k,m\ge
    k} \Prb(\#A'=l,\ \#B'= m\ |\ \#U=l+m)
\end{align*}
(the final line follows since $\sum_{l\ge k,m\ge k} \Prb(\#U=l+m)\le 1$).
\goodbreak

We have $|A'|\le|A|\le |C|\le|C'|$ so $|A'|\le \frac12|U|$ and similarly
$B'\le \frac12|U|$. Hence,
for $l,m\ge k$,
\begin{align*}
  \Prb(\#A'=l,\ \#B'= m\ |\ \#U=l+m)
  =& \binom{l+m}{l}\left(\frac {|A'|}{|U|}\right)^l
  \left(\frac {|B'|}{|U|}\right)^m\\
&\le 2^{l+m}\left(\frac {|A'|}{|U|}\right)^l
  \left(\frac {|B'|}{|U|}\right)^m\\
&\le 2^{2k}\left(\frac {|A'|}{|U|}\right)^k
  \left(\frac {|B'|}{|U|}\right)^k\\
&= \left(\frac {4|A'||B'|}{|U|^2}\right)^{k}\\
&= \left(\frac {4|A'||B'|}{(|A'|+|B'|+|C'|)^2}\right)^{k}.
\end{align*}
Finally, observe  that $|A'|\le|A|\le |C|\le |C'|$ and
$|B'|\le|B|\le |C|\le |C'|$ imply that
\begin{align*}
\frac {4|A'||B'|}{(|A'|+|B'|+|C'|)^2}&\le \frac
      {4|A||B'|}{(|A|+|B'|+|C'|)^2}\\
&\le \frac {4|A||B|}{(|A|+|B|+|C'|)^2}\\
&\le \frac {4|A||B|}{(|A|+|B|+|C|)^2}.
\end{align*}
which completes the proof.\end{proof}
\subsection{Proof of Theorem~\ref{t:centre}}
By hypothesis we have $k>0.4125\log n$. Also, we may assume that
$k<0.6 \log n$ since we already know that $G_{n,k}$ is connected whp
if $k\ge 0.6 \log n$.  Let $c'=\max\{c_1(0.25,1),c_2(0.25,1),1\}$ be
as given by Lemmas~\ref{l:small} and~\ref{l:edges} and let
$M=20000c'$. (We shall reuse some of the bounds we prove here in the
proof of Theorem~\ref{t:boundary} so these are convenient values.)
Tile $S_n$ with small squares of side length $s=\sqrt{\log n}/M$.  We
form a graph $\widehat G$ on these tiles by joining two tiles whenever
the distance between their centres is at most $2c'\sqrt{\log n}$.  We
call a pointset $\cP$ \emph{bad} if any of the following hold:
\begin{enumerate}
\item there exist two points that are joined in $G$ but  the tiles containing these
points are not joined in $\widehat G$,
\item there exist two points, at most distance $20000s$ apart, that are
  not joined,
\item there exists a half-disc based at a point of $\cP$ of radius
  $c'\sqrt{\log n}$ that is contained entirely in $S_n$ and contains no
  (other) point of $\cP$,
\item there exist two components in $G_{n,k}$ with Euclidean diameter
  at least $c'\sqrt{\log n}$,
\item there exists a component of diameter at most $c'\sqrt{\log n}$
  containing a vertex within distance $2c'\sqrt{\log n}$ of the
  boundary of $S_n$,
\end{enumerate}
and \emph{good} otherwise.  We see that our choice of $c'$ and $M$
together with Lemma~\ref{l:edges} imply that the probability that any
of the first three conditions occur is $O(n^{-1})$. By
Lemma~\ref{l:small} the probability of the fourth condition is
$O(n^{-1})$.  Since $k>0.4125\log n>\frac{1}{\log 25}\log n$,
Lemma~\ref{l:boundary} implies the probability of the last condition
is $O(n^{-\eps})$ for some $0<\eps<1$. (Alternatively this follows from
Theorem~\ref{t:boundary}).  Combining these we see that the
probability of a bad configuration is $O(n^{-\eps})$.

Suppose that $\cP$ is a good configuration but $G$ is not
connected. Then there exists a component $F$ with diameter at most
$c'\sqrt{\log n}$ not containing any vertex within $2c'\sqrt{\log n}$
of the boundary of~$S_n$. Let $A$ be the collection of tiles that
contain a point of $F$. Since the configuration is good $A$ is a
connected subset of $\widehat G$ containing no tile within
$c'\sqrt{\log n}$ of the boundary of~$S_n$. Moreover, the bound on the
diameter of $F$ implies that $A$ contains at most $16(c'M)^2$ tiles.

The heart of the proof is in the following lemma that bounds the
probability of $G$ having such a component.
\begin{lemma}\label{l:centre}
Suppose $A$ is a connected subset of $\widehat G$ 
containing no tile within $c'\sqrt{\log n}$ of the boundary
of~$S_n$. The probability that the configuration is good and that
$G$ has a component contained entirely inside $A$ meeting every tile
of $A$ is at most $O(11.3^{-k})$.
\end{lemma}
\begin{proof}
Suppose that $F$ is a component of
$G$ meeting every tile in $A$.

The proof of this lemma naturally divides into three steps. In the
first step we define some regions based on the component $F$ some of
which must contain many points and some which must be empty. In the
second step we bound the area of these regions. In the final step we
bound the probability that these regions do indeed contain the
required number of points.

\begin{figure}
\begin{center}
\includegraphics[width=250pt]{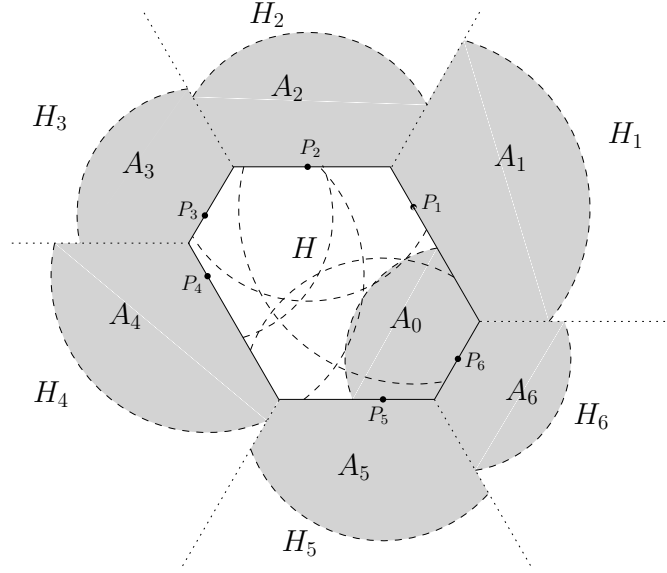}
\end{center}
\caption{The circumscribed hexagon $H$ and associated
   regions.}\label{f:upper1}
\end{figure}

{\it Step 1: Defining the regions.}  We use the following hexagonal
construction which was introduced by Balister, Bollob\'as, Sarkar and
Walters in~\cite{MR2135151}.  Let $H$ be the circumscribed hexagon of
the points of $F$ obtained by taking the six tangents to the convex
hull of $F$ at angles $0$ and $\pm 60\degree$ to the horizontal,and
let $H_1,\ldots, H_6$ be the regions bounded by the exterior angle
bisectors of $H$ as in Figure~\ref{f:upper1}.  Let $P_1,\ldots,P_6$ be
the points of $F$ on these tangents, and let $D_1,\ldots,D_6$ denote
the $k$-nearest neighbour disks of $P_1,\ldots,P_6$. For $1\le i\le 6$
let $A_i=D_i\cap H_i$. Let $A_0$ be the set $D_i\cap H$ with the
smallest area. We see that for each $1\le i\le 6$ the set $A_i$
contains no points of $\cP$. Also $A_0$ contains $k+1$ points all of
which must be in $F$ and thus in $A$.  Writing $A'$ for the set
$A_0\cap A$, we see that $A'$ contains at least $k+1$ points of $\cP$.

We also wish to take account of points near to but not contained in
$F$. Let $P\in F$ and $Q\in G\setminus F$ be vertices minimising the
distance between $F$ and $G\setminus F$.  Let $r_0=d(P,Q)$ and
$r=r_0-\sqrt2 s$.  Since, we are assuming that
every square of $A$ contains a point in $F$ 
we see that $\ar \setminus A$ contains no point of $\cP$. Indeed,
suppose there is a point of $\cP$ in $\ar \setminus A$. Then this point
is in $G\setminus F$ and is within $r_0$ of some point of $F$ which
contradicts the definition of $r_0$.

Obviously the points $Q$ and $P$ are not joined so, in particular, the
$k$ points nearest to $Q$ must all be nearer to $Q$ than $P$
is. Moreover, since $Q$ is the point closest to $F$, we see that these
$k$ points must all be further away from $P$ than $Q$ is.  Combining
these we see that these $k$ points lie in in the set
$B=D(Q,r_0)\setminus D(P,r_0)$.

Summarising all of the above, we see that $A'$ and $B$ each contain at
least $k$ points and $\ar \setminus A$ and $\bigcup_{i=1}^6 A_i$ are
both empty. The intersection $A'\cap B$ contains no points (so we can
think of them as disjoint) but $\ar $ and $\bigcup_{i=1}^6 A_i$ will
overlap significantly. Thus we will use Lemma~\ref{l:small} to form
two separate bounds, one based on $\ar \setminus A$ being empty and one
based on $\bigcup_{i=1}^6 A_i$ being empty.

\bigskip\noindent%
{\it Step 2: Bounding the area of the regions.}
In this step we assume that the configuration is good.

First we bound $|\bigcup_{i=1}^6 A_i|$. Since the configuration is
good each disc $D_i$ has radius at most $c'\sqrt{\log n}$ and each
point $P_i$ is more than $2c'\sqrt{\log n}$ from the boundary of
$S_n$. In particular $D_i$ is contained in $S_n$ for each
$i$. Moreover, since $|D_i\cap H_i|\ge |D_i\cap H|$ for each $1\le
i\le 6$, we see that $|A_i|\ge |A_0|$. Since the $H_i$ and
therefore the $A_i$ are disjoint, we have
\[
|\bigcup_{i=1}^6 A_i|\ge 6|A_0|\ge 6|A'|.
\]

The sets $B$ and $\ar $ both depend on $r$ so it is convenient to write
$r$ in terms of $|A'|$ by letting $x=r/(\sqrt{|A'|/\pi})$. 

Since $B=D(Q,r_0)\setminus D(P,r_0)$ a simple calculation shows that
$|B|=\left(\frac{\pi}{3}+\frac{\sqrt
  3}{2}\right)r_0^2$.  Since the configuration is good,
$r_0>20000s$ so
\[
  r=r_0-\sqrt2 s> r_0(1-10^{-4}).
\]
Hence,
\[
|B|=\left(\frac{\pi}{3}+\frac{\sqrt
  3}{2}\right)r_0^2 \le
\left(\frac{\pi}{3}+\frac{\sqrt
  3}{2}\right)\frac{x^2|A'|}{\pi(1-10^{-4})^2}< 0.61x^2|A'|.
\]

Finally we bound $\ar $. Let $D$ and $D'$ be balls of area $|A|$ and
$|A'|$ respectively. Since the configuration is good the the half-disc
of radius $c'\sqrt{\log n}$ about the right-most point of $F$ must
contain a point of $\cP$. In particular $r<r_0\le c'\sqrt{\log n}$,
and so $\ar $ is contained in $S_n$.  By the isoperimetric inequality
in the plane
\[
|\ar \setminus A|\ge |\dr\setminus D|,
\] 
and it easy to see that $|\dr\setminus D|\ge |\dr'\setminus D'|$.
Since $D'$ is a ball of radius $\sqrt{|A'|/\pi}$, $\dr'$ is a ball of
radius $\sqrt{|A'|/\pi}+r=(1+x)\sqrt{|A'|/\pi}$, and we have
\[
|\dr'\setminus D'|=((x+1)^2-1)|A'|.
\]
 
\bigskip
{\it Step 3: Bounding the probability of such a configuration.}  We
have seen that if there is such a component $F$ then there exist
regions as defined in Step 1. These regions are determined by 14 points:
the six points defining sides of the hexagonal hull, their six $k^{\rm
  th}$ nearest neighbour points and the points $P$ and $Q$; that is,
if there is such a component $F$ then there are 14 points of $\cP$
defining regions $A'$, $B$, $A_1,\ldots, A_6$ and $\ar $ with $\#A'\ge
k$, $\#B\ge k$, $\#(A'\cap B)=0$, and both $\#\bigcup_{i=1}^6A_i=0$
and $\#(\ar \setminus A)=0$. Moreover, if the configuration is good all
of these points must lie within $c'\sqrt{\log n}$ of $A$.

Let $Z$ be the event that there are 14
points of $\cP$ all within $c'\sqrt{\log n}$ of $A$ defining regions with the
above properties.  We have
\begin{align*}
\Prb(\text{there exists } F& \text{ and the configuration is good})\\
&\le\Prb( \text{Z and the configuration is good})\\&\le \Prb(Z).
\end{align*}

We bound the probability that $Z$ occurs (note we are not assuming
that the configuration is good). Fix a particular collection of 14
points of $\cP$ and let $Z'$ be the event that these particular points
witness $Z$. Note, since we are assuming these 14 points all lie with
$c'\sqrt{\log n}$ of $A$, the corresponding regions all lie entirely
within $S_n$.

We apply Lemma~\ref{l:kk} to the sets $A'$, $B$ together
with each of $\bigcup_{i=1}^6A_i$ and $\ar \setminus A$. 

First we form the bound based on $\#\bigcup_{i=1}^6A_i=0$. We have
$|A'|\le |\bigcup_{i=1}^6A_i|$ and, provided $x<3.13$, we have $|B|\le
0.61x^2<6|A'|\le |\bigcup_{i=1}^6A_i|$ so Lemma~\ref{l:kk}
applies. Thus we see that
\[
\Prb(Z')\le \left(\frac{4|A'||B|}{(|A'|+|(\bigcup_{i=1}^6A_i)|+|B|)^2}\right)^k
\le \left(\frac{4\cdot0.61x^2}{(7+0.61x^2)^2}\right)^k.
\]

Secondly we form a bound based on $\#(\ar \setminus A)=0$.  This time
$|B|\le 0.61 x^2|A'|\le ((x+1)^2-1)|A'|\le |\ar \setminus A|$ and,
provided that $x>\sqrt2-1$, we have $|\ar \setminus A|\ge |A'|$ so the
conditions of Lemma~\ref{l:kk} are satisfied. Thus
\[
\Prb(Z')\le \left(\frac{4|A'||B|}{(|A'|+|\ar \setminus A|+|B|)^2}\right)^k \le
\left(\frac{4\cdot0.61x^2}{((x+1)^2+0.61x^2)^2}\right)^k.
\]

It is easy to check that the maximum of the minimum of these two
bounds occurs when they are equal, i.e.,  when $x=\sqrt7-1$; at this point
they are $\alpha^{-k}$ for some $\alpha>11.3$. Therefore $\Prb(Z')\le
\alpha^{-k}$.

Since all 14 points must lie within $c'\sqrt{\log n}$ of $A$ there are
$O((\log n)^{14})$ ways of choosing them. Hence, the expected number
of 14 point sets for which $Z'$ occurs is is $O((\log
n)^{14}\alpha^{-k})=O(11.3^{-k})$. Thus $\Prb(Z)=O(11.3^{-k})$ and
the proof of the lemma is complete.
\end{proof}

Since the degree of vertices in $\widehat G$ is bounded and
$16(c'M)^2$ is a (large) constant, there are only a constant number of
connected sets of $\widehat G$ of size at most $16(c'M)^2$ which
contain a fixed tile, and therefore $O(n)$ such sets in total. Since
$k>0.4125\log n>\frac{1}{\log 11.3}\log n$ the expected number of small
components in $G$ with the configuration good is
$O(n(11.3)^k)=o(1)$. Thus
\begin{align*}
\Prb(\text{$G$ is not}&\text{ connected})\\&\le\Prb(\text{there is a
  small component and $\cP$ is good})+\Prb(\text{$\cP$ is
  bad})\\ &=o(1)+O(n^{-\eps})\\&=o(1),
\end{align*}
 so whp $G$ is connected.\qed

\subsection{Proof of Theorem~\ref{t:boundary}}
Much of this is the same as the proof of Theorem~\ref{t:centre} so we
shall concentrate on the differences.  This time, by hypothesis we
have $k>0.272\log n$ and again we may assume $k<0.6\log n$. We use
exactly the same tesselation of $S_n$ with small squares of side
length $s=\sqrt{\log n}/M$ where $c'=\max\{c_1(0.25,1),c_2(0.25,1)\}$
and $M=20000c'$ are given by Lemmas~\ref{l:small} and~\ref{l:edges} as
before.  Again we form a graph $\widehat G$ on these tiles by joining
two tiles whenever the distance between their centres is at most
$2c'\sqrt{\log n}$.

We need a slightly different definition of a bad pointset: the first four
conditions are exactly as before but we replace the fifth condtion by
\begin{enumerate}

\item[5.] there exists a component of diameter at most $c'\sqrt{\log n}$
  containing a vertex within distance $3c'\sqrt{\log n}$ of two sides
  of  $S_n$.
\end{enumerate}
Note that this condition, together with Condition~4 on the diameter of
small components, implies that for any small component at most one
side of $S_n$ can have points of this small component within distance
$2c'\sqrt{\log n}$ of it.

Since the tesselation is the same as in the proof of
Theorem~\ref{t:centre} we see that the probability that any of the
original four conditions hold is $O(n^{-1})$ as before.  Since
$k>0.272\log n$ Lemma~\ref{l:boundary} implies that the probability of
the new condition above is $O(n^{-\eps})$ for some
$0<\eps<1$. Combining these we see that the probability of a bad
configuration is $O(n^{-\eps})$.

Suppose that $\cP$ is a good configuration but not all points within
$\log n$ of the boundary of $S_n$ are contained in the giant
component. Then there exists a component $F$ with diameter at most
$c'\sqrt{\log n}$ containing a vertex within $\log n$ of the boundary
of~$S_n$. Let $A$ be the collection of tiles that contain a point of
$F$. Since the configuration is good $A$ is a connected subset of
$\widehat G$ and, as before, the bound on the diameter of $F$ implies
that $A$ contains at most $16(c'M)^2$ tiles. This time at most one
side of $S_n$ has any tiles of $A$ within $c'\sqrt{\log n}$ of it.

The following lemma, which is similar to Lemma~\ref{l:centre} bounds
the probability of such a small component.
\begin{lemma}
Suppose $A$ is a connected subset of $\widehat G$ such that at most
one side of $S_n$ has any tiles in $A$ within $c'\sqrt{\log n}$ of
it. The probability that the configuration is good and that $G$ has a
small component contained entirely inside $A$ which meets every square
of $A$ is at most $(6.3)^{-k}$.
\end{lemma}
\begin{remark}
  Obviously this lemma is only of interest for sets $A$ near the
  boundary, since otherwise Lemma~\ref{l:centre} is stronger.
\end{remark}
\begin{proof}
The proof divides into the same three steps as Lemma~\ref{l:centre}.

\bigskip\noindent%
{\it Step 1: Defining the regions.}  As before
suppose that $F$ is a component of $G$ meeting every tile in $A$. Let
$E$ be the (almost surely unique) side of $S_n$ closest to $F$.

\begin{figure}
\begin{center}
\includegraphics[width=250pt]{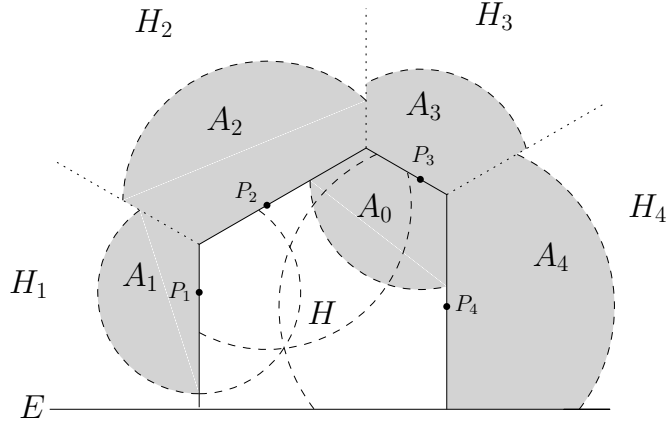}
\end{center}
 \caption{The circumscribing set $H$ and associated
   regions.}\label{f:upper2}
\end{figure}

This time let $H$ be the region bounded by the four interior sides of
the circumscribed hexagon of the points of $F$ obtained by taking four of
the tangents to the convex hull of $F$ at angles $90\degree$ and $\pm
30\degree$ to $E$, together with $E$ as in
Figure~\ref{f:upper2}.  Let $H_1,\ldots, H_4$ be the regions bounded
by the exterior angle bisectors of $H$ and $E$.  Let $P_1,\ldots,P_4$
be the points of $F$ on these tangents, and let $D_1,\ldots,D_4$ denote the
$k$-nearest neighbour disks of $P_1,\ldots,P_4$. For $1\le i\le 4$ let
$A_i=D_i\cap H_i$. Let $A_0$ be the set $D_i\cap H$ with the smallest
area and write $A'$ for the set $A_0\cap A$.  Exactly as before we see
that for $1\le i\le 4$ the set $A_i$ is empty, and that $A'$ must
contain at least $k+1$ points of $\cP$.

As before let $P\in F$ and $Q\in G\setminus F$ be vertices minimising
the distance between $F$ and $G\setminus F$, $r_0=d(P,Q)$ and
$r=r_0-\sqrt2 s$. Again, since $F$ meets every tile of $A$ we see that
$\ar \setminus A$ must be empty. Also, as before, the set
$B=(D(Q,r_0)\setminus D(P,r_0))\cap S_n$ must contain at least $k$
points.

\bigskip\noindent%
{\it Step 2: Bounding the area of the regions.} In this step we
assume the configuration is good.

First we bound $|\bigcup_{i=1}^4 A_i|$. Similarly to before we see that each
disc $D_i$ has radius at most $c'\sqrt{\log n}$ so meets no side of
$S_n$ apart from possibly $E$. Thus, we have $|D_i\cap H_i|\ge
|D_i\cap H|$ for each $1\le i\le 4$, so we see that $|A_i|\ge
|A_0|$. As before the $H_i$ and therefore the $A_i$ are disjoint so
\[
|\bigcup_{i=1}^4 A_i|\ge 4|A_0|\ge 4|A'|.
\]

As before let $x=r/\sqrt{|A'|/\pi}$ and exactly as in the proof of
Lemma~\ref{l:centre} we have $|B|<0.61x^2|A'|$. 

Finally we bound $\ar \setminus A$.  Consider the point of $F$ furthest
from $E$ and the half disc of radius $c'\sqrt{\log n}$
about that point facing away from $E$. Since no point of $F$ is within
$c'\sqrt{\log n}$ of any side of $S_n$ apart from $E$, this half disc
is entirely inside $S_n$, and so must contain a
point of $\cP$ (which is obviously not in $F$). Therefore, as before,
$r<r_0\le c'\sqrt{\log n}$. Thus $\ar \cap S_n=\ar \cap E_+$ where $E_+$
denotes the halfplane bounded by $E$ that contains $S_n$.

This time let $D$ and $D'$ be half discs of area $|A|$ and $|A'|$
respectively centred on $E$. Then, by the isoperimetric inequality in the
half plane $E_+$ (an easy consequence of the same inequality in the whole
plane),
\[
|(\ar \cap E_+) \setminus A|\ge |(\dr\cap E_+)\setminus D|\ge |(\dr'\cap E_+)\setminus D'|.
\] 
Now $D'$ is half a disc of radius $\sqrt2 \sqrt{|A'|/\pi}$ and
$\dr'\cap E_+$
is half a disc of radius $\sqrt2
\sqrt{|A'|/\pi}+r=(1+x/\sqrt2)\sqrt{2|A'|/\pi}$, so this time we we
have
\[
 |(\dr'\cap E_+)\setminus D'|=((1+x/\sqrt2)^2-1)|A'|.
\]

\bigskip\noindent%
{\it Step 3: Bounding the probability of such a configuration.}  We
have seen that if there is such a component $F$ then there exist
regions as defined above. These regions are determined by 10 points:
the four points defining sides of the hexagonal hull, their four
$k^{\rm th}$ nearest neighbour points and the points $P$ and $Q$; that
is, if there is such a component $F$ then there are 10 points of $\cP$
defining regions $A'$, $B$, $A_1,\ldots, A_4$ and $\ar $ with $\#A'\ge
k$, $\#B\ge k$, $\#(A'\cap B)=0$, and both $\#\bigcup_{i=1}^4A_i=0$
and $\#((\ar\cap S_n) \setminus A)=0$. Again, if the configuration is good, all
these points must lie within $c'\sqrt{\log n}$ of $A$.

Similarly to before, let $Z$ be the event that there are 10
points of $\cP$ all within $c'\sqrt{\log n}$ of $A$ defining regions with the
above properties.  Again
\begin{align*}
\Prb(\text{there exists } F& \text{ and the configuration is good})\\
&\le\Prb( \text{Z and the configuration is good})\\&\le \Prb(Z)
\end{align*}
so, as before, we bound $\Prb(Z)$.

Fix a particular collection of 10 points and let $Z'$ be the event
that these 10 points witness $Z$.  Note, since we are assuming these
10 points all lie with $c'\sqrt{\log n}$ of $A$, the regions
$A',A_1,\ldots,A_4$ all lie entirely within $S_n$. By definition, $B$ and
$(\ar\cap S_n)\setminus A$ also lie in $S_n$.

Again we apply Lemma~\ref{l:kk} to the sets $A'$, $B$ together with
each of $\bigcup_{i=1}^4A_i$ and $(\ar\cap S_n) \setminus A$. This time,
however, neither bound will be valid for large $x$ so we form a third
bound based just on the two sets $A'$ and $(\ar\cap S_n) \setminus A $.

As before we base the first bound on $\#\bigcup_{i=1}^4A_i=0$. We have
$|A'|\le |\bigcup_{i=1}^4A_i|$ and, provided $x<2.56$, we have $|B|\le
0.61x^2|A'|<4|A'|\le |\bigcup_{i=1}^4A_i|$ so Lemma~\ref{l:kk} implies
\[
\Prb(Z')\le \left(\frac{4|A'||B|}{(|A'|+|(\bigcup_{i=1}^4A_i)|+|B|)^2}\right)^k
\le \left(\frac{4\cdot0.61x^2}{(5+0.61x^2)^2}\right)^k.
\]

The second bound based on $\#((\ar\cap S_n) \setminus A)=0$ is also very similar
to before.  However, this time the middle inequality in 
\[
|B|\le 0.61 x^2|A'|\le ((1+x/\sqrt 2)^2-1)|A'|\le
|(\ar\cap S_n) \setminus A|
\] is not valid for all $x$, but it is valid for all $x<12$.
Also provided that $x>2-\sqrt2$, we have $|(\ar\cap S_n) \setminus A|\ge |A'|$ so
for $2-\sqrt 2<x<12$ the conditions of Lemma~\ref{l:kk} are
satisfied. Thus
\[
\Prb(Z')\le \left(\frac{4|A'||B|}{(|A'|+|(\ar\cap S_n) \setminus A|+|B|)^2}\right)^k \le
\left(\frac{4\cdot0.61x^2}{((1+x/\sqrt2)^2+0.61x^2)^2}\right)^k.
\] 

Since neither bound applies for large $x$ we form a third bound based
on the two sets $A'$ and $(\ar\cap S_n) \setminus A$. We know $A'$ contains at
least $k$ points and $(\ar\cap S_n) \setminus A$ is empty. This has probability
at most
\[
\Prb(Z')\le \left(\frac{|A'|}{|A'|+|(\ar\cap S_n) \setminus A|}\right)^k\le
\frac{1}{(1+x/\sqrt 2)^{2k}}
\]
which is less than $80^{-k}$ for all $x\ge 12$.

As before the maximum of the minimum of the first two bounds occurs
when they are equal at $x=\sqrt2(\sqrt 5-1)$; at this point they are
$\alpha^{-k}$ for some $\alpha>6.3$. Moreover the third bound is tiny
in comparison. Thus, in all cases, $\Prb(Z')\le \alpha^{-k}$ for some $\alpha>6.3$.

Since all 10 points must lie within $c'\sqrt{\log n}$ of $A$ there are
$O((\log n)^{10})$ ways of choosing them. Hence, similarly to before,
the expected number of 10 point sets for which $Z'$ occurs is is
$O((\log n)^{10}\alpha^{-k})=O(6.3^{-k})$. Hence
$\Prb(Z)=O(6.3^{-k})$ and the proof of the lemma is complete.
\end{proof}

The remainder of the proof is very similar to before. There are only a
constant number of connected sets of $\widehat G$ of size at most
$16(c'M)^2$ which contain a fixed tile, and therefore $O(\sqrt{n\log
  n})$ such sets which contain a tile within distance $\log n$ of the
boundary of $S_n$.  Since $k>0.272\log n> \frac{1+\eps'}{\log
  6.3}\log(\sqrt {n})$ for some $\eps'>0$ the expected
number of small components of $G$ that contain a vertex within
distance $\log n$ of the boundary of $S_n$ when the configuration is
good is $O(\sqrt {n\log n}(6.3)^{-k})=o(n^{-\eps'/2})$. Let
$\eps=\min(\eps'/2,1)$ and $p$ be the the probability that there
exists a point $\cP$ within $\log n$ of the boundary of $S_n$ that is
not in the giant component. Then
\begin{align*}
p&\le\Prb(\text{there is a small boundary component and $\cP$ is
  good})+\Prb(\text{$\cP$ is bad})\\
&=o(n^{-\eps})+O(n^{-\eps})\\&=O(n^{-\eps})
\end{align*}
as claimed.\qed
\subsection*{Open Questions}
In this paper we have proved two results about the behaviour of the
small components in the graph $G_{n,k}$. However, several question
about their properties remain open. We are interested in the behaviour
near the connectivity threshold so, in particular, we assume in the 
following questions that $k$ is at least $0.3\log n$.
\begin{question}
  Must the small components of $G_{n,k}$ be isolated? More precisely,
  is it the case that, whp, there do not exist two small components
  within distance of $O(\sqrt{\log n})$ of each other.
\end{question}
\noindent%

Since the first draft of this paper
Falgas-Ravry~\cite{Falgas-poissondistr} has answered this question in
the affirmative provided that the probability that $G$ is connected is
not too small: more precisely he proves it whenever
$\Prb(G$~is~connected$)=\Omega(n^{\gamma})$ (where $\gamma$ is an
absolute constant).

\begin{question}
  How many vertices do small components contain? 
\end{question}
\noindent%
It is immediate from Lemma~6 of~\cite{MR2135151} (quoted as
Lemma~\ref{l:small} of this paper) that all small components contain
$O(k)$ vertices.  If the lower bound construction of Balister,
Bollob\'as, Sarkar and Walters in~\cite{MR2135151} is extremal then,
as the authors remark there, all small components would contain
$k+O(1)$ vertices.

\begin{question}
  Are all the small components convex in the sense that all points of
  $\cP$ within the convex hull of a small component are actually part
  of the small component?
\end{question}


\bibliography{mybib}{}
\bibliographystyle{abbrv}

\end{document}